\documentclass[a4paper,11pt]{amsart}
\usepackage{mathptmx}
\usepackage{amsmath}
\usepackage{amscd}
\usepackage{amssymb, bm}
\usepackage{amsthm}
\usepackage{xspace}
\usepackage[all,tips]{xy}
\usepackage[dvips]{graphicx}
\usepackage{verbatim}
\usepackage{syntonly}
\usepackage{hyperref}
\usepackage{graphics, fullpage,color, epsfig,url}
\usepackage{indentfirst}
\usepackage{esint}
\usepackage{enumitem}
\usepackage[dvipsnames]{xcolor}
\usepackage{soul, amsaddr}
\usepackage{tensor}
\usepackage{amsrefs, eucal}
\usepackage{bbm}
\usepackage{ulem}

\providecommand{\MR}{\relax\ifhmode\unskip\space\fi MR }

\providecommand{\href}[2]{#2}

\theoremstyle{plain}
\newtheorem{thm}{Theorem}[section]
\newtheorem{lem}[thm]{Lemma}
\newtheorem{prop}[thm]{Proposition}
\newtheorem{defn}[thm]{Definition}
\newtheorem{cor}[thm]{Corollary}

\newtheorem*{ex}{Example}

\theoremstyle{remark}
\newtheorem{rem}[thm]{Remark}
\newtheorem*{rems}{Remark}

\newcommand{\disp}{\displaystyle}

\DeclareMathOperator{\supp}{supp}
\DeclareMathOperator{\di}{div}

\DeclareMathOperator{\loc}{loc}

\newcommand{\eps}{\varepsilon}
\newcommand{\vp}{\varphi}

\newcommand{\al}{\alpha}
\newcommand{\be}{\beta}
\newcommand{\ga}{\gamma}
\newcommand{\de}{\delta}

\newcommand{\la}{\lambda}
\newcommand{\La}{\Lambda}

\newcommand{\iny}{\infty}

\newcommand{\su}{\subset}
\newcommand{\LP}{\Delta}
\newcommand{\gr}{\nabla}

\newcommand{\inrn}{\ensuremath{\int_{\R^n}}}

\newcommand{\norm}[1]{\left\| #1\right\|}

\newcommand{\abs}[1]{\left\vert#1\right\vert}

\newcommand{\set}[1]{\left\{#1\right\}}
\newcommand{\brac}[1]{\left[#1\right]}

\newcommand{\pr}[1]{\left( #1 \right) }
\newcommand{\pb}[1]{\left( #1 \right] }

\newcommand{\N}{\ensuremath{\mathbb{N}}}

\newcommand{\R}{\ensuremath{\mathbb{R}}}

\newcommand{\C}{\ensuremath{\mathbb{C}}}

\newcommand{\Rn}{\ensuremath{\R^n}}

\def\XXint#1#2#3{{\setbox0=\hbox{$#1{#2#3}{\int}$}
\vcenter{\hbox{$#2#3$}}\kern-.5\wd0}}

\numberwithin{equation}{section}
\date{}

\begin{document}

\title{Unique continuation at infinity for Schr\"odinger equations \\ with Reverse H\"older Potentials}
\author[Davey]{Blair Davey}
\address{Department of Mathematical Sciences, Montana State University, Bozeman, MT, 59717}
\email{blairdavey@montana.edu}
\thanks{Davey is supported in part by National Science Foundation CAREER DMS - 2236491}

\begin{abstract}
In this article, we study unique continuation properties at infinity for solutions to generalized Schr\"odinger equations with potential functions that belong to the reverse H\"older class.
For equations of the form $-\di \pr{A \gr u} + V u = 0$ in $\R^n$, where $A$ is bounded and elliptic and $V \in RH_p$ for some $p \in [\frac n 2, \iny]$, we prove that if a solution doesn't grow too quickly at infinity, then it must be trivial.
We use $d_V$, the Agmon distance function associated to $V$, to quantify the threshold growth rate.
More precisely, there exists a constant $\ga_0 > 0$ so that if $|u(x)| \lesssim \exp\pr{\ga \, d_V(x, 0)}$ for some $\ga < \ga_0$ and every $x \in \R^n$, then $u$ must be trivial.
The result may be interpreted as a Liouville-type theorem and is related to Landis' conjecture.
Our proof techniques are inspired by Z. Shen's exponential decay estimates for fundamental solutions of Schr\"odinger operators and involve the application of a Fefferman-Phong inequality.
\end{abstract}

\maketitle

\section{Introduction}

In the late 1960s, E.~M.~Landis \cite{KL88} made the following conjecture regarding the global unique continuation properties of solutions to Schr\"odinger equations: 
If $u$ is a bounded solution to 
\begin{equation}
\label{ePDE0}
-\LP u + V u = 0 \; \text{ in } \, \R^n,
\end{equation}
where $V$ is a bounded function and $u$ satisfies $\abs{u(x)} \lesssim \exp\pr{- c \abs{x}^{1+}}$, then $u \equiv 0$.
This conjecture was later disproved by Meshkov \cite{Mes92} who constructed non-trivial $\C$-valued functions $u$ and $V$ that satisfy $-\LP u + V u = 0$ in $\R^2$, where $V$ is bounded and $\abs{u(x)} \lesssim \exp\pr{- c \abs{x}^{4/3}}$. 
Meshkov also proved a \textit{qualitative unique continuation} result: 
If $u$ solves \eqref{ePDE0}, where $V$ is bounded and $u$ satisfies a decay estimate of the form $\abs{u\pr{x}} \lesssim \exp\pr{- c \abs{x}^{4/3+}}$, then necessarily $u \equiv 0$.
In their work on Anderson localization \cite{BK05}, Bourgain and Kenig established a quantitative version of Meshkov's result. 
As a first step in their proof, they used three-ball inequalities derived from Carleman estimates to establish \textit{order of vanishing} estimates (quantifications of the strong unique continuation property) for local solutions to Schr\"odinger equations.
Then, through a scaling argument, they proved a \textit{quantitative unique continuation} result.
More specifically, they proved that if $u$ and $V$ are bounded, and $u$ is normalized so that $\abs{u(0)} \ge 1$, then for sufficiently large values of $R$,
\begin{equation}
\label{est}
 \inf_{|x_0| = R}\norm{u}_{L^\iny\pr{B_1(x_0)}} \ge \exp{(-CR^{\be}\log^\mu R)},
\end{equation} 
where $\be = \frac 4 3$ and $\mu = 1$.
The bound in \eqref{est} may be described as an estimate for the \textit{rate of decay at infinity}.
Since $ \frac 4 3 > 1$, the constructions of Meshkov, in combination with the qualitative and quantitative unique continuation theorems just described, indicate that Landis' conjecture cannot be true in the complex-valued setting.

In the real-valued setting, the remarkable work of Logunov, Malinnikova, Nadirashvili, and Nazarov \cite{LMNN25} established a quantitative version of Landis' conjecture in the planar setting, $\R^2$.
That is, if $u$ and $V$ are bounded and real-valued, and $u$ is appropriately normalized, then for sufficiently large values of $R$, \eqref{est} holds with $\be = 1$ and $\mu = \frac 3 2$.
Very recently, Frank and Ivanisvili \cite{FI26} constructed real-valued functions $u , V: \R^n \to \R$ that solve \eqref{ePDE0} with $V$ bounded and $\abs{u(x)} \lesssim \exp\pr{- c \abs{x}^{4/3}}$ for each $x \in \R^n$.
As their examples work for all $n \ge 3$, the real-valued version of Landis' conjecture is now completely understood.

In this article, we study the qualitative unique continuation properties of real-valued solutions to variable-coefficient Schr\"odinger equations with reverse H\"older potentials.
Accordingly, the questions that we address here are related to Landis' conjecture.
However, given that reverse H\"older functions are signed and unbounded, our setting is distinct from the classical one.
Moreover, in contrast to the work of \cite{BK05} and \cite{LMNN25}, for example, we don't formulate our estimates in terms of the Euclidean distance to the origin.
Here we take a different perspective and derive results in terms of the \textit{Agmon distance}.
As articulated in Definition \ref{AgmonDistance}, the Agmon distance, $d_V$, is naturally associated to a reverse H\"older function $V \in RH_p$ whenever $p \ge \frac n 2$. 

In our main result, we show that if a solution doesn't have enough growth at infinity, then it must be trivial.
In some sense, this theorem is similar to Meshkov's qualitative unique continuation result \cite{Mes92} described above which states that if a solution decays too quickly, then it must be trivial.
Given that we consider $V \in RH_p$, the maximum principle is applicable, so it is overly restrictive to consider solutions that decay at infinity (or are even bounded).
In other words, Meshkov's result is immediate for solutions to equations with non-negative potentials.
Therefore, in our setting, it is quite natural to examine the ``optimal growth at infinity".
We use the Agmon distance function to quantify the threshold growth function.

We formulate the main result in a few ways.
Our first theorem states that if there exists a collection of concentric annuli with bounded, increasing radii over which some exponentially weighted integrals of the solution are diminishing, then the solution must be trivial.
We choose to work with annuli since they are most naturally suited to the proof techniques.
In the second version of the main result, we show that if the solution is pointwise bounded by an exponentially growing function, then it must be trivial.
We include this formulation because it is very similar to a Liouville-type statement.
The last version of the main result shows if $u$ is a non-trivial solution, then there exists a sequence of points over which the norm of $u$ is exponentially increasing with respect to the Agmon distance.
Although this version of the result is a technical consequence of a first (qualitative) statement, this particular formulation was inspired by the quantitative unique continuation results described by \eqref{est}.

Now we make the setting precise.
We consider weak solutions to variable-coefficient Schr\"odinger equations of the form
\begin{equation}
\label{ePDE}
- \di \pr{A \gr u} + V u = 0 \; \text{ in } \, \R^n,
\end{equation}
where the coefficient matrix $A : \R^n \to \mathbb{M}_{n}$ is bounded and elliptic and the potential function $V : \R^n \to \R$ belongs to the reverse H\"older class, $RH_p$, for some $p \in \brac{\frac n 2, \iny}$.
By \textit{bounded} and \textit{elliptic}, we mean that for some $0 < \la \le 1 \le \La < \iny$, it holds that
\begin{align}
&\abs{A(x) \xi \cdot \zeta} \le \La \abs{\xi} \abs{\zeta} \, \text{ for a.e. } x \in \R^n, \xi, \zeta \in \R^n .
\label{Abound} \\
&A(x) \xi \cdot \xi \ge \la \abs{\xi}^2 \, \text{ for a.e. } x \in \R^n, \xi \in \R^n .
\label{ellip} 
\end{align}
We say that $V$ is a \textit{reverse H\"older function}, $V \in RH_p$, if $V \in L^p_{\text{loc}}(\R^n)$, $V > 0$ a.e., and there exists $C_V > 0$ so that for every ball $B \su \R^n$,
$$\frac 1 {C_V} \pr{\fint_B V^p}^{\frac 1 p} \le \fint_B V := \frac 1 {\abs{B}} \int_B V.$$
A function $u \in H^1_{\loc}(\R^n)$ is a \textit{weak solution} to \eqref{ePDE} if for every $\vp \in H^1_0(\R^n)$ with compact support, it holds that
\begin{equation}
\label{weakSol}
\int_{\R^n} A \gr u \cdot \gr \vp + V u \vp = 0.
\end{equation}
Since $V \in RH_p$ implies that $V$ is real-valued, there is no loss of generality in assuming that the solution function $u$ is also real-valued.

We use $B_r(x_0)$ to denote a ball of radius $r > 0$ centered at a point $x_0 \in \R^n$.
If the center is understood from context, or we take $x_0 = 0$, then we may write $B_r$ for short.
We let $A(x_0, r_1, r_2) = B_{r_2}(x_0) \setminus \overline{B_{r_1}(x_0)}$ denote the open annulus centered at $x_0$ with inner radius $r_1 > 0$ and outer radius $r_2 > r_1$.
When the center is understood, we may abbreviate this to $A(r_1, r_2)$.

The main theorem of this article is as follows.
As mentioned above, this statement may be interpreted as a qualitative unique continuation result.

\begin{thm}[Main result]
\label{MainThm}
Let $A : \R^n \to \mathbb{M}_{n}$ be bounded and elliptic in the sense of \eqref{Abound} and \eqref{ellip}.
For $p \in \brac{\frac n 2, \iny}$, let $V \in RH_p$ and let $d_V : \R^n \times \R^n \to \R_{\ge 0}$ denote the Agmon distance function.
There exists a constant $\ga_0(n, \la, \La, p, C_V) > 0$ so that the following holds:

If $u : \R^n \to \R$ a weak solution to \eqref{ePDE} and there is an increasing, unbounded sequence $\set{R_j}_{j = 1}^\iny \su \R_+$ for which
\begin{equation}
\label{uIntBound}
\lim_{j \to \iny} \int_{A\pr{R_j, R_j + 1}} \abs{u(x)}^2 e^{-2 \ga d_V(x, 0)} dx = 0
\end{equation}
for some $\ga < \ga_0$, then $u \equiv 0$.
\end{thm}

Instead of formulating the growth condition with integrals, we may instead impose a pointwise bound on the solution to obtain a statement that resembles Liouville's theorem.

\begin{cor}[Liouville-type result]
\label{LiouvilleCor}
Let $A : \R^n \to \mathbb{M}_{n}$ be bounded and elliptic in the sense of \eqref{Abound} and \eqref{ellip}.
For $p \in \brac{\frac n 2, \iny}$, let $V \in RH_p$ and let $d_V : \R^n \times \R^n \to \R_{\ge 0}$ denote the Agmon distance function.
There exists a constant $\ga_0(n, \la, \La, p, C_V) > 0$ so that the following holds:

If $u : \R^n \to \R$ is a weak solution to \eqref{ePDE} and $u$ has bounded growth in the sense that for some $\ga_1 < \ga_0$, there exists $C_{\ga_1} > 0$ so that for each $x \in \R^n$,
\begin{equation}
\label{uPwBound}
\abs{u(x)} \le C_{\ga_1} e^{\ga_1 d_V(x, 0)},
\end{equation}
then $u \equiv 0$.
\end{cor}

Corollary \ref{LiouvilleCor} follows from Theorem \ref{MainThm} since \eqref{uPwBound} implies \eqref{uIntBound} with any $\ga \in (\ga_1, \ga_0)$ and $R_j = j$, for example.
The details of this argument appear in Section \ref{MainResultProofs} after the proof of Theorem \ref{MainThm}.

Finally, we state a consequence of the contrapositive to Theorem \ref{MainThm}.
This version of the result was inspired by the quantitative unique continuation results described in \eqref{est}, for example.

\begin{cor}[Growth at infinity]
\label{OGICor}
Let $A : \R^n \to \mathbb{M}_{n}$ be bounded and elliptic in the sense of \eqref{Abound} and \eqref{ellip}.
For $p \in \brac{\frac n 2, \iny}$, let $V \in RH_p$ and let $d_V : \R^n \times \R^n \to \R_{\ge 0}$ denote the Agmon distance function.
There exists a constant $\ga_0(n, \la, \La, p, C_V) > 0$ so that the following holds:

If $u : \R^n \to \R$ is a non-trivial weak solution to \eqref{ePDE}, then for every $\ga_1 < \ga_0$, there exists an increasing, unbounded sequence of points $\set{x_j}_{j = 1}^\iny \su \R^n$ with 
\begin{equation}
\label{uballBound}
\norm{u}_{L^2\pr{B_1(x_j)}}  \ge j e^{2 \ga_1 d_V(x_j, 0)} 
\end{equation}
for every $j \in \N$.
\end{cor}

Corollary \ref{OGICor} follows from the contrapositive to Theorem \ref{MainThm} in combination with a delicate treatment of the Agmon distance function.
These details appear at the end of Section \ref{MainResultProofs}.

Given that $V(x) \equiv V_0 \in \R_+$ belongs to $RH_p$ for every $p > 1$, simplified versions of these results hold for solutions to $-\LP u + V_0 u = 0$ in $\R^n$.
We discuss these simpler statements and their proofs in Section \ref{exampleSection}.

Schr\"odinger operators with reverse H\"older potentials have been studied in numerous contexts.
Shen proved solvability of the Neumann problem for $-\LP + V$ with $V \in RH_\iny$ \cite{She94}, various $L^p$ estimates for $-\LP + V$ with $V \in RH_p$ \cite{She95}, eigenvalue asymptotics and exponential decay for eigenfunctions of magnetic Schr\"odinger operators with reverse H\"older potentials \cite{She96}, and two-sided exponential decay estimates for fundamental solutions of generalized Schr\"odinger operators \cite{She99}.
Accordingly, much of the machinery that was developed by Shen in \cite{She94, She95, She96, She99} is used in this article.
Shen's work in \cite{She95} was extended by Auscher and Ben Ali in \cite{AB07} to include a larger range of $p$-values, and later generalized to magnetic operators by Ben-Ali in \cite{Ben10}.
Mayboroda and Poggi proved upper and lower exponential decay estimates for fundamental solutions of very general magnetic Schr\"odinger operators in \cite{MP19}.
In the systems setting, fundamental matrices associated to generalized Schr\"odinger operators with reverse H\"older potentials were studied by the author with Mayboroda and Hill in \cite{DHM18} and with Isralowitz in \cite{DI24}.
The article \cite{DI24} includes two-sided exponential decay estimates for fundamental matrices.

The author is not aware of other articles that explore the unique continuation properties of time-independent Schr\"odinger equations with bounded, elliptic coefficients and reverse H\"older potentials.
However, there are some unique continuation results that are formulated in terms of the Agmon distance.
In \cite{Ueb26}, Uebersch{\"a}r studies unique continuation at infinity for solutions to Schr\"odinger equations with radial potential functions and shows that under suitable smoothness assumptions, the decay threshold can be formulated in terms of the Agmon distance.
We also point out that Uebersch{\"a}r's Liouville-type result in \cite{Ueb25} (which apply to solutions to Schr\"odinger equation with continuous non-negative potentials) are similar in flavor to Theorem \ref{MainThm}.

Most unique continuation results (at least in dimension $n \ge 3$) require some continuity assumptions on the coefficient matrix $A$.
In the local setting, the examples of Plis \cite{Pli63}, Miller \cite{Mil74}, Mandache \cite{Man98}, and Filonov \cite{Fil01} show that unique continuation can fail for elliptic equations with H\"older continuous coefficients.
The globally-defined solutions constructed in \cite{FLP25} and \cite{DJ26} decay super-exponentially and therefore demonstrate the possible failure of unique continuation at infinity for elliptic equations with Lipschitz continuous coefficients that don't decay fast enough.
Thus, it is somewhat surprising that Theorem \ref{MainThm} holds under the minimal assumptions that $A$ is merely bounded and elliptic.
One interpretation of our result is that the reverse H\"older potential has some kind of smoothing effect. 

The strong unique continuation property (SUCP) holds for $-\LP + V$ whenever $V \in L^p_{\loc}$ and $p > \frac n 2$, \cite{JK85}. 
Since $RH_p \su L^p_{\loc}$, the SUCP automatically holds in our setting if we assume that $A$ is also symmetric and Lipschitz continuous.
Because of its scale-invariance properties, the reverse H\"older class may be more suited for global problems.
However, it could be interesting to investigate the local unique continuation properties for the operators that we consider here.

Most rate of decay estimates (Landis-type theorems) and their related order of vanishing estimates (quantifications of the SUCP) are proved using Carleman estimates \cite{Mes92}, \cite{BK05}, \cite{CS99}, \cite{Dav14} or frequency functions \cite{Kuk98}, \cite{Zhu16}, \cite{Dav27}.
In the planar setting, complex analytic techniques have lead to breakthrough results \cite{KSW15}, \cite{LMNN25}, but these tools are not applicable in the current setting.

The techniques employed in this article are inspired by the ideas in \cite{She99} (see also \cite{MP19} and \cite{DI24}) that were used to prove exponential upper and lower bounds of fundamental solutions.
More specifically, our main tool, Lemma \ref{MainUpperBoundLem}, is very similar to \cite[Lemma 3.1]{She99}, \cite[Proposition 4.3, Proposition 6.5]{MP19}, and \cite[Lemma 20]{DI24}.
To prove Lemma \ref{MainUpperBoundLem}, we rely on the framework developed by Shen in \cite{She99}, for example.
This framework includes the ideas that are used to develop the notion of the Agmon distance, see Definition \ref{AgmonDistance}, as well as the Fefferman-Phong inequality, Proposition \ref{FPml}, and Agmon's argument from \cite{Agm82}.
The Fefferman-Phong inequality \cite{Fef83} was originally developed in the study of the uncertainty principle, but our version resembles that in \cite[Lemma 1.11]{She94}.
By combining Proposition \ref{FPml} with integral identities that hold for weak solutions to elliptic equations, we arrive at Lemma \ref{MainUpperBoundLem}.
With an appropriate choice of cutoff function, we derive a sequence of inequalities that imply Theorem \ref{MainThm}.
In some sense, Lemma \ref{MainUpperBoundLem} plays the role of a Carleman estimate in our arguments.

To demonstrate the central ideas in the proof of Theorem \ref{MainThm}, a simplified version of the main theorem is presented in Theorem \ref{MainThm0}.
In addition, a simpler version of the main lemma appears in Lemma \ref{MainUpperBoundLem0}.
The method that we use to go from Lemma \ref{MainUpperBoundLem0} to Theorem \ref{MainThm0} is a streamlined version of that which takes us from Lemma \ref{MainUpperBoundLem} to Theorem \ref{MainThm}, but it still captures the significant ideas.

This article is organized as follows.
In Section \ref{exampleSection}, we state and prove a version of the main theorem and its first corollary for the equation $-\LP u + V_0 u = 0$ in $\R^n$, where $V_0 > 0$ is constant.
The intention is that this section demonstrates the big ideas behind the proof of Theorem \ref{MainThm}.
Section \ref{RHetc} introduces the reverse H\"older class and develops the theory that leads to the definition of the Agmon distance.
The content of this section is mostly reproduced from the works of Shen in \cite{She94}, \cite{She95}, \cite{She96}, and \cite{She99}.
In Section \ref{FPI}, we state a known Poincar\'e inequality, we state and prove a (slight modification of a) known Fefferman-Phong inequality, and then we use these results to give our main lemma, Lemma \ref{MainUpperBoundLem}.
As mentioned above, Lemma \ref{MainUpperBoundLem} is inspired by \cite[Lemma 3.1]{She99}, \cite[Proposition 4.3, Proposition 6.5]{MP19}, and \cite[Lemma 20]{DI24}.
The proofs of Theorem \ref{MainThm} and its corollaries appear in the final section, Section \ref{MainResultProofs}.

\section{Illustrative results}
\label{exampleSection}

To demonstrate the main ideas in the proof of Theorem \ref{MainThm}, we consider weak solutions to equations of the form $- \LP u + V_0 u = 0$ where $V_0$ is a positive constant.
A function $u \in H^1_{\loc}(\R^n)$ is a weak solution to $- \LP u + V_0 u = 0$ in $\R^n$ if for every $\vp \in H^1_0(\R^n)$ with compact support, it holds that
\begin{equation}
\label{weakEigen}
\int_{\R^n} \gr u \cdot \gr \vp + V_0 u \vp = 0.
\end{equation}
Although weak solutions to $- \LP u + V_0 u = 0$ are also classical solutions, we use the weak framework as preparation for future sections.

The goal of this section is to prove the following statement.

\begin{thm}[Example main result]
\label{MainThm0}
If $u : \R^n \to \R$ is a weak solution to  $- \LP u + V_0 u = 0$ in $\R^n$ and there exists an increasing, unbounded sequence $\set{R_j}_{j = 1}^\iny \su \R_+$ with 
\begin{equation}
\label{uIntBound0}
\lim_{j \to \iny} \int_{A\pr{R_j, R_j + 1}} \abs{u(x)}^2 e^{-2 \ga \sqrt{V_0} \abs{x}} dx = 0
\end{equation}
for some $\ga < 1$, then $u \equiv 0$.
\end{thm}

A consequence of Theorem \ref{MainThm0} is the following Liouville-type statement for solutions to $-\LP u + V_0 u = 0$ in $\R^n$.
To see that Theorem \ref{MainThm0} implies Corollary \ref{LiouvilleCor0}, take $\ga = \frac{1 + \ga_1}{2}$ and $R_j = j$ for each $j \in \N$.

\begin{cor}[Example Liouville-type result]
\label{LiouvilleCor0}
Let $u \in H^1_{\loc}(\R^n)$ be a weak solution to $- \LP u + V_0 u = 0$ in $\R^n$.
If $u$ has bounded growth in the sense that for some $\ga_1 \in (0, 1)$, there exists $C_{\ga_1} > 0$ so that for each $x \in \R^n$,
\begin{equation}
\label{uBound0}
\abs{u(x)} \le C_{\ga_1} e^{\ga_1 \sqrt{V_0} \abs{x}},
\end{equation}
then $u \equiv 0$.
\end{cor}

To understand Corollary \ref{LiouvilleCor0}, we consider some examples.
When $n = 1$, all non-trivial solutions to $- u'' + V_0 u = 0$ are of the form $u(x) = c e^{\pm \sqrt{V_0} x}$ and these fail to satisfy the assumption \eqref{uBound0}.
When $n = 2$, the separable solutions to $- \LP u + V_0 u = 0$ are
$$u_z(x_1, x_2) = \exp\pr{\pm \sqrt{V_0 - z} \, x_1 \pm \sqrt z \, x_2},$$
where $z \in \R$.
When $z \in (-\iny, 0]$, $u_z$ fails condition \eqref{uBound0} in the $\pm x_1$-direction; while if $z \in [V_0, \iny)$, $u_z$ fails condition \eqref{uBound0} in the $\pm x_2$-direction.
Even if $z \in (0, V_0)$, $u_z$ fails condition \eqref{uBound0} along one of the lines $\disp x_1 =\pm \frac{\sqrt{V_0 - z}}{\sqrt z} x_2$.
In higher dimensions, we can construct similar examples that support the Liouville-type statement above.

If we restrict to exterior domains, the Liouville-type statement above is no longer true.
For example, $u(x) = c e^{- \sqrt{V_0} \abs{x}}$ solves $- u'' + V_0 u = 0$ in $\R \setminus (-1,1)$ and is exponentially decaying.
In the punctured plane, the modified Bessel function of the second kind with $\al = 0$ gives a radial, exponentially decreasing solution to $- \LP u + V_0 u = 0$ in $\R^2 \setminus B_1$.
Accordingly, it is necessary for us to consider global solutions.

To prove Theorem \ref{MainThm0}, we rely on the following lemma.

\begin{lem}[Example main lemma]
\label{MainUpperBoundLem0}
Let $u \in H^1_{\loc}(\R^n)$ be a weak solution to $- \LP u + V_0 u = 0$ in $\R^n$.
If $\phi \in C_c ^\infty (\R^n)$, then for any $\ga < 1$, there exists a constant $C_\ga > 0$ so that
\begin{align*}
\int_{\R^n} \abs{\sqrt{V_0} \, u(x) \, \phi(x) }^2 e^{- 2 \ga \sqrt{V_0 \abs{x}^2 + 1} } dx
&\le C_\ga \int_{\R^n} \abs{u(x)}^2 \abs{\gr \phi(x)}^2 e^{-2 \ga \sqrt{V_0 \abs{x}^2 + 1}} dx.
\end{align*}
\end{lem}

\begin{proof}
Set $g(x) = \sqrt{V_0 \abs{x}^2 + 1}$.
For $\phi \in C_c ^\infty (\R^n)$, $\ga < 1$, let $f(x) = \phi(x) e^{-\ga g(x)}$.
Set $\vp = u f^2$ and note that $\vp \in H^1_0(\R^n)$ with compact support equal to $\supp \phi$.
Set $h = u f$ and note that $h \in H^1_0(\R^n)$ with compact support as well.
Moreover, $h^2 = u \vp$ and since
\begin{align*}
\abs{\gr h}^2
&= \abs{f \gr u + u \gr f}^2
= f^2 \abs{\gr u}^2 + u^2 \abs{\gr f}^2 + 2 u f \gr u \cdot \gr f  \\
\gr u \cdot \gr \vp
&= \gr u \cdot \gr\pr{u f^2}
= \gr u \cdot \pr{f^2\gr u + 2 u f \gr f}
= f^2 \abs{\gr u}^2 + 2 u f \gr u \cdot \gr f,
\end{align*}
then
\begin{align*}
\abs{\gr h}^2
&= \gr u \cdot \gr \vp
+ u^2 \abs{\gr f}^2.
\end{align*}
Now observe that
\begin{align*}
V_0 \int \abs{u \phi }^2 e^{- 2 \ga g}
&= \int V_0 \abs{h}^2
\le \int \abs{\gr h}^2 + V_0 \abs{h}^2
= \int \gr u \cdot \gr \vp + V_0 u \vp
+ \int u^2 \abs{\gr f}^2 
=  \int u^2 \abs{\gr f}^2,
\end{align*}
where we have used the equation \eqref{weakEigen}. 
Since 
$$\gr f = \pr{\gr \phi - \ga \sqrt{V_0} \phi \frac{x}{ \sqrt{x^2 + V_0^{-1}}} } e^{- \ga g},$$
then for any $\mu > 0$, 
$$\abs{\gr f}^2 \le \brac{\pr{1 + \mu^{-1}} \abs{\gr \phi}^2 + \pr{1 + \mu}\ga^2 V_0 \abs{\phi}^2} e^{- 2\ga g}$$
so that
\begin{align*}
V_0 \int \abs{u \phi }^2 e^{- 2 \ga g}
&\le \pr{1 + \mu^{-1}} \int u^2 \abs{\gr \phi}^2 e^{- 2 \ga g}
+ \pr{1 + \mu}\ga^2 V_0 \int \abs{ u \phi}^2 e^{- 2 \ga g}.
\end{align*}
Since $\ga < 1$, then we may choose $\mu > 0$ so that $\pr{1 + \mu}\ga^2 = \frac{1 + \ga^2}{2} < 1$.
Absorbing the second term on the right into the left and simplifying shows that
\begin{align*}
V_0 \int \abs{ u \phi}^2 e^{- 2 \ga g}
&\le 2 \frac{1 + \ga^2}{\pr{1 - \ga^2}^2} \int u^2 \abs{\gr \phi}^2 e^{- 2 \ga g}
\le \pr{\frac{2}{1 - \ga^2}}^2 \int u^2 \abs{\gr \phi}^2 e^{- 2 \ga g},
\end{align*}
as required.
\end{proof}

With this version of the main lemma, we now prove our simplified version of the main theorem.

\begin{proof}[Proof of Theorem \ref{MainThm0}]
Assume for the sake of contradiction that $u \not\equiv 0$ on $\R^n$.
In this case, there exists $R_0 \ge 1$ so that with $B_{R_0} := B_{R_0}(0)$, 
\begin{equation}
\label{nontrivial0}
\norm{u}_{L^2(B_{R_0})} = c_0 > 0.
\end{equation}

Let $\set{R_j}_{j = 1}^\iny \su \R_+$ be an unbounded, increasing sequence for which \eqref{uIntBound0} holds with $\ga < 1$.
Without loss of generality, $R_j \ge R_0$ for each $j \in \N$.

For each $j \in \N$, let $\phi_j \in C_c ^\infty(B_{R_j+1}(0))$ be a cutoff function with $\phi_j \equiv 1$ in the ball $B_{R_j} := B_{R_j}(0)$ and $|\nabla \phi_j | \leq 2$ in the annulus $A(R_j, R_j + 1) := A(0, R_{j}, R_{j} + 1)$.

With $\phi = \phi_j$ and $g(x) := \sqrt{V_0 \abs{x}^2 + 1}$, we apply Lemma \ref{MainUpperBoundLem0} to get
\begin{align*}
V_0 \int_{B_{R_j}} \abs{u }^2 e^{- 2 \ga g}
&\le \int_{\R^n} \abs{\sqrt{V_0} u \phi_j }^2 e^{- 2 \ga g}
\le C_\ga \int_{\R^n} \abs{u}^2 \abs{\gr \phi_j}^2 e^{-2 \ga g} 
\le 4 C_\ga \int_{A(R_j, R_j + 1)} \abs{u}^2 e^{-2 \ga \sqrt{V_0} \abs{x}}.
\end{align*}
Since $R_j \ge R_0$, we get
\begin{equation*}
\label{lowerBound}
\begin{aligned}
\int_{B_{R_j}} \abs{u }^2 e^{- 2 \ga g}
\ge \int_{B_{R_0}} \abs{u }^2 e^{- 2 \ga g}
\ge e^{- 2 \ga g(R_0)} \int_{B_{R_0}} \abs{u(x)}^2 
= e^{- 2 \ga \sqrt{V_0 R_0^2 + 1}} c_0^2,
\end{aligned}
\end{equation*}
where we have used the assumption in \eqref{nontrivial0}.
That is, for each $j \in \N$.
\begin{align*}
\frac{V_0 c_0^2}{4 C_\ga  e^{2 \ga \sqrt{V_0 R_0^2 + 1}} }
&\le \int_{A(R_j, R_j + 1)} \abs{u}^2 e^{-2 \ga \sqrt{V_0} \abs{x}}.
\end{align*}
By taking $j \to \iny$ and applying \eqref{uIntBound0}, we conclude that $c_0 = 0$.
In other words, we reach the desired contradiction and the conclusion follows.
\end{proof}

\section{The Reverse H\"older class and Agmon distance}
\label{RHetc}

In this section, we introduce the reverse H\"older functions, the class of potential functions that we consider.
We state a number of the properties of reverse H\"older functions that lead us to define the Agmon distance.
The concepts that we develop in this section were introduced by Zhongwei Shen.
Specifically, what we call the averaged function $\psi_V$ and the auxiliary function $m_V$ were introduced in \cite{She94} and further developed in \cite{She95, She96, She99}, while the distance function $d_V$ appears in \cite{She96, She99}. 

\begin{defn}[$RH_p$ functions]
\label{BpDefn}
A function $V : \R^n \to \R_{\ge 0}$ belongs to the \textbf{reverse H\"older class} $RH_p$ if $V \in L^p_{\loc}(\R^n)$, $V > 0$ a.e., and there exists a constant $C_V$ so that for every ball $B \su \R^n$,
\begin{equation}
\label{BpDefOne}
\begin{cases}
\disp \pr{\fint_B \brac{V\pr{x}}^p dx}^{\frac 1 p} \le C_V \fint_B V\pr{x} dx & p \in (1, \iny) \\
\disp \norm{V}_{L^\iny(B)} \le C_V \fint_B V\pr{x} dx & p = \iny
\end{cases}.
\end{equation}
We call $C_V = C_{V,p}$ the \textbf{uniform $RH_p$ constant for $V$}.
\end{defn}

\begin{rems}
While the zero function satisfies \eqref{BpDefOne}, we assume that $V > 0$ a.e. to exclude trivial functions from the $RH_p$ classes.
\end{rems}

\begin{ex}[Reverse H\"older functions] 
$\quad$
\begin{itemize}
\item[(a)] If $P(x)$ is a polynomial and $\al > 0$, then $\abs{P}^\al \in RH_p$ for every $p > 1$.
\item[(b)] $V(x) := \abs{x}^\al \in RH_{\frac n 2}$ whenever $\al > -2$.
\end{itemize}
\end{ex}

We first state Gehring's Lemma, a self-improvement result.
For the proof of Gehring's Lemma, see for example the dyadic approach in \cite{Per01}.

\begin{lem}[Gehring's Lemma]
\label{GehringLemma}
If $V \in RH_p$, then there exists $\eps(p, C_V) > 0$ so that $V \in RH_{p+\eps}$.
In particular, $V \in RH_q$ for all $q \in \brac{1, p + \eps}$.
Moreover, if $q \le s$, then $C_{V, q} \le C_{V, s}$.
\end{lem}

Since $RH_p \su A_\iny$, the reverse H\"older functions are doubling.

\begin{lem}[Doubling result]
\label{Vdbl}
If $V \in RH_p$, then $V$ is a doubling measure.
That is, there exists a doubling constant $\de = \de\pr{n, p, C_V} > 0$ so that for every $x \in \R^n$ and every $r > 0$,
\begin{align*}
\int_{B_{2r}(x)} V\pr{y} dy \le \de \int_{B_r(x)} V\pr{y} dy.
\end{align*}
\end{lem}

\begin{defn}[Averaged function, cf. (1.1) in \cite{She94}]
Given a function $V : \R^n \to \R$, its {\bf averaged function} $\psi_V : \R^n \times \R_+ \to \R$ is defined as
\begin{align}
\psi_V(x, r) = r^2 \fint_{B_r(x)} V\pr{y} dy.
\label{eqB.2}
\end{align}
\end{defn}

Observe that the averaged function respects scaling in the following way:
If $V_R(y) := R^2 V(x_0 + Ry)$, then 
\begin{align*}
\psi_{V_R}(0, r)
&= r^2 \fint_{B_r(0)} V_R\pr{y} dy
= r^2 \fint_{B_r(0)} R^2 V(x_0 + Ry) dy
= \psi_V(x_0, rR). 
\end{align*}
Since $\psi_{V_R}(0, r) = \psi_V(x_0, rR)$, where $V_R$ is the natural rescaling of $V$, then we may think of $\psi_V$ as having a ``scale-invariant" property.

\begin{lem}[Controlled growth, cf. Lemma 1.2 in \cite{She95}]
\label{BasicShenLem}
If $V \in RH_p$, then for any $0 < r < R  < \iny$,
\begin{align*}
\psi_V(x, r) \le C_V \pr{\frac{r}{R}}^{2 - \frac{n}{p}} \psi_V(x, R),
\end{align*}
\label{lB.1}
where $C_V$ is the uniform $RH_p$ constant for $V$.
\end{lem}

In the statement of Theorem \ref{MainThm}, we assume that $V \in RH_p$ for some $p \in \brac{\frac n 2, \iny}$.
Gehring's Lemma shows that there is no loss in assuming that $p > \frac n 2$.
By the doubling property and that $V > 0$ a.e., $V$ is nondegenerate in the sense that for every open set $U \su \R^n$, $\disp \int_U V(y) dy > 0$.
Therefore, $\psi_V(x, r) > 0$ for each $x \in \R^n$ and $r > 0$.
Since $2 - \frac n p > 0$, then it follows from Lemma \ref{lB.1} that
\begin{equation}
\label{eqB.3}
\begin{aligned}
& \lim_{r \to 0^+} \psi_V(x, r) = 0 \\
& \lim_{R \to \iny} \psi_V(x, R) = \iny.
\end{aligned}
\end{equation}
These observations allow us to make the following definition of $m_V$, the auxiliary function.

\begin{defn}[Auxiliary function, cf. (1.3) in \cite{She94}]
\label{auxFunc}
Let $V \in RH_p$ for some $p \in (\frac n 2, \iny]$.
The {\bf auxiliary function} $m_V : \Rn \rightarrow  (0, \infty)$ is defined as
\begin{align}
\label{mDef}
\frac{1}{m_V(x)} = \sup_{r > 0} \set{ r :\psi_V(x, r) \le 1}.
\end{align}
\end{defn}

The auxiliary function has a number of very useful properties.
These ideas appear in \cite{She94}, \cite{She95}, \cite{She99}, and \cite{MP19}, for example.
We recall the following lemma from \cite{She95}.

\begin{lem}[cf. Lemma 1.4, \cite{She95}]
\label{mBounds}
If $V \in RH_p$ for some $p \in (\frac n 2, \iny]$, then there exist constants $c_1, c_2, c_3, k_0 > 0$, depending on $n$, $p$, and $C_V$, so that for any $x, y \in \R^n$,
\begin{enumerate}
\item[(a)] If $\disp \abs{x - y} \le \frac{2}{m_V(x)}$, then $\disp c_1^{-1} m_V(x) \le  m_V(y) \le c_1 m_V(x)$;
\item[(b)] $\disp m_V(y) \le c_2 \pr{1 + \abs{x - y}m_V(x)}^{k_0} m_V(x)$;
\item[(c)] $\disp m_V(y) \ge \frac{c_3 \, m_V(x)}{\pr{1 + \abs{x -y}m_V(x)}^{k_0/\pr{k_0+1}}}$.
\end{enumerate}
\end{lem}

Using the auxiliary function, we now define the associated Agmon distance function.

\begin{defn}[Agmon distance]
\label{AgmonDistance}
Let $V \in RH_p$ for some $p \in (\frac n 2, \iny]$.
The \textbf{Agmon distance function} $d_V : \R^n \times \R^n \to \R_{\ge 0}$ is defined as
\begin{equation*}
d_V(x, y) = \inf_{\ga} \int_0^1 m_V(\ga(t)) |\ga^{\;\prime}(t)|\, dt ,
\end{equation*}
where the infimum ranges over all absolutely continuous $\ga:[0,1] \to \R^n$ with $\ga(0) = x$ and $\ga(1) = y$.
\end{defn}

As observed in \cite[eq. (3.19), (3.22)]{She99}, applications of Lemma \ref{mBounds} lead to the following bounds on the Agmon distance function.

\begin{lem}[Agmon distance bounds]
\label{ulDistance}
If $V \in RH_p$ for some $p \in (\frac n 2, \iny]$, then there exist constants $c_4, c_5 > 0$, depending on $n$, $p$, and $C_V$, so that for any $x, y \in \R^n$,
\begin{enumerate}
\item[(a)] If $\disp \abs{x - y} \le \frac{1}{m_V(x)}$, then $d_V(x, y) \le c_1$;
\item[(b)] If $\disp \abs{x - y} \ge \frac{1}{m_V(x)}$, then $\disp d_V(x, y) \le c_4 \pr{\abs{x - y} m_V(x)}^{k_0+1}$;
\item[(c)] If $\disp \abs{x - y} \ge \frac{1}{m_V(x)}$, then $\disp d_V(x, y) \ge c_5 \pr{\abs{y - x} m_V(x) }^{1/(k_0+1)}$.
\end{enumerate}
The constants $c_1$ and $k_0$ are from Lemma \ref{mBounds}.
\end{lem}

\begin{proof}
For $x, y \in \R^n$, define $\ga : \brac{0,1} \to \R^n$ to be the straight line path given by $\ga(t) = x + t\pr{y - x}$. \\
If $\abs{x - y} \le \frac{1}{m_V(x)}$, Lemma \ref{mBounds}(a) shows that $m_V(\ga(t)) \le c_1 m_V(x)$ for all $t \in \brac{0, 1}$ and it follows that
\begin{align*}
d_V(x, y)
&\le \int_0^1 m_V(\ga(t)) \abs{\ga^{\;\prime}(t)} dt
\le c_1 \int_0^1 m_V(x) \abs{x - y} dt
\le c_1,
\end{align*}
leading to (a).

If $\abs{x - y} \ge \frac{1}{m_V(x)}$, Lemma \ref{mBounds}(b) shows that $m_V(\ga(t)) \le c_2 \pr{1 + t \abs{x - y}m_V(x)}^{k_0} m_V(x)$ and then
\begin{align*}
d_V(x, y)
&\le c_2 \int_0^1 \pr{1 + t \abs{x - y}m_V(x)}^{k_0} m_V(x) \abs{x - y} dt
= c_2 \int_1^{1 + \abs{x - y}m_V(x)} z^{k_0}dz \\
&= \frac{c_2}{k_0 + 1} \brac{\pr{1 + \abs{x - y}m_V(x)}^{k_0+1} - 1}
\le \frac{c_2}{k_0 + 1} \pr{2 \abs{x - y}m_V(x)}^{k_0+1},
\end{align*}
and the conclusion (b) follows. 

Choose $\ga : \brac{0,1} \to \R^n$ so that $\ga\pr{0} = x$, $\ga\pr{1} = y$ and
\begin{align*}
d_V(x, y) \ge \frac 1 2 \int_0^1 m_V(\ga(t)) \abs{\ga^{\;\prime}(t)} dt.
\end{align*}
It follows from Lemma \ref{mBounds}(c) that
\begin{align*}
d_V(x, y)
\ge \frac {c_3} 2 \int_0^1 \frac{m_V(x)\abs{\ga^{\,\prime}(t)} dt}{\pr{1 + \abs{\ga(t) - x} m_V(x)}^{k_0/(k_0+1)}}
= \frac {c_3} 2 \int_0^1 \frac{\abs{\widetilde \ga \,'(t)} dt}{\pr{1 + \abs{\widetilde \ga(t)}}^{k_0/(k_0+1)}},
\end{align*}
where $\widetilde \ga : \brac{0,1} \to \R^n$ is defined as $\tilde \ga(t) = m_V(x) \pr{\ga(t) - x}$, a shifted, rescaled version of $\ga$ with $\widetilde \ga(0) = 0$ and $\widetilde \ga(1) = m_V(x) \pr{y - x}$.
This integral is bounded from below by the geodesic distance from $0$ to $m_V(x) \pr{y - x}$ in the metric
$$\frac{dz}{\pr{1 + \abs{z}}^{k_0/(k_0+1)}}.$$
A computation shows that the straight line path achieves this minimum and then
\begin{align*}
d_V(x, y)
&\ge \frac {c_3} 2 \int_0^1 \frac{m_V(x) \abs{y - x} dt}{\pr{1 + m_V(x) t \abs{y - x}}^{k_0/(k_0+1)}}
= \frac {c_3 (k_0+1)} 2 \brac{ \pr{1+m_V(x) \abs{y - x}}^{1/(k_0+1)} - 1}.
\end{align*}
Since $\abs{x - y} \ge \frac 1 {m_V(x)}$, then a mean value argument leads to the conclusion (c).
\end{proof}

In future sections, the Agmon distance function will be an important tool for us once it has been suitably regularized as is done in \cite{She99}.
Observe that by Theorem \ref{mBounds}(c), $m_V$ is a slowly varying function in the sense of \cite[Definition 1.4.7]{Hor83}.
As such, we have the following result.

\begin{lem}[cf. the proof of Lemma 3.3. in \cite{She99}]
\label{partofU}
Let $V \in RH_p$ for some $p \in (\frac n 2, \iny]$.
There exist sequences $\set{x_j}_{j=1}^\iny \su \R^n$ and $\set{\phi_j}_{j=1}^\iny \su C^\iny_0(\R^n)$, and constants $c_6, c_7 > 0$, depending on $n$, $p$, and $C_V$, so that
\begin{itemize}
\item[(a)] $\disp \R^n = \bigcup_{j=1}^\iny B_j$, where $r_j := m_V(x_j)^{-1}$, $\disp B_j := B_{r_j}(x_j)$;
\item[(b)] $\phi_j \in C^\iny_0\pr{B_j}$, $0 \le \phi_j \le 1$, and $\disp \sum_{j=1}^\iny \phi_j = 1$;
\item[(c)] $\abs{\gr \phi_j\pr{x}} \le c_6 \, m_V(x)$;
\item[(d)] $\disp \sum_{j=1}^\iny \chi_{B_j} \le c_7.$
\end{itemize}
\end{lem}

The previous lemma and \cite[Theorem 1.4.10]{Hor83} were used in \cite[pp. 542]{She99} to establish the following result.
For completeness, we include the proof below.

\begin{lem}[Lemma 3.3 in \cite{She99}]
\label{RegLem0}
Let $V \in RH_p$ for some $p \in (\frac n 2, \iny]$.
There exists a constant $c_8(n, p, C_V) >  0$ so that for every $y \in \R^n$, there exists a nonnegative function $g_V(\cdot, y) \in C^\infty(\R^n)$ such that for every $x \in \R^n$,
$$\abs{g_V(x, y) - d_V(x, y)} \le c_1$$
and
$$|\gr_x \, g_V(x, y)| \le c_8 \, m_V(x).$$
The constant $c_1$ is from Lemma \ref{mBounds}.
\end{lem}

\begin{proof}
Let $\set{x_j}_{j=1}^\iny$, $\set{B_j}_{j = 1}^\iny$, and $\set{\phi_j}_{j=1}^\iny$ be from Lemma \ref{partofU}.
For each $y \in \R^n$, define 
$$g_V(x, y) = \sum_j d_V(x_j, y) \phi_j(x)= \sum_{j: x \in B_j} d_V(x_j, y) \phi_j(x).$$
Since $\disp \sum_j \phi_j = 1$, then
\begin{align*}
g_V(x, y) - d_V(x, y)
&= \sum_j \brac{d_V(x_j, y) - d_V(x, y)} \phi_j(x) 
\end{align*}
and
\begin{align*}
\abs{g_V(x, y) - d_V(x, y)}
&\le \sum_{j: x \in B_j} \abs{ d_V(x_j, y) - d_V(x, y) } \phi_j(x) 
\le \sum_{j: x \in B_j} d_V(x_j, x) \phi_j(x) 
\le \sum_{j} c_1 \phi_j(x) 
= c_1,
\end{align*}
where we have applied the triangle inequality, Lemma \ref{ulDistance}(a), then Lemma \ref{partofU}(b).

For the second part, we differentiate to get
\begin{align*}
\gr_x g_V(x, y)
&= \sum_{j} d_V(x_j, y) \gr \phi_j(x) 
= \sum_{j} \brac{d_V(x_j, y) - d_V(x, y)} \gr \phi_j(x) +  d_V(x, y) \sum_{j} \gr \phi_j(x) \\
&= \sum_{j} \brac{d_V(x_j, y) - d_V(x, y)} \gr \phi_j(x),
\end{align*}
since $\disp 0 = \gr 1 = \sum_j \gr \phi_j$.
Taking norms and using the triangle inequality then shows that
\begin{align*}
\abs{\gr_x g_V(x, y)}
&\le \sum_{j: x \in B_j} \abs{d_V(x_j, y) - d_V(x, y)} \abs{\gr \phi_j(x) }
\le \sum_{j: x \in B_j} d_V(x_j, x) \abs{\gr \phi_j(x) }
\le \sum_{j: x \in B_j} c_1 c_6 m_V(x) \\
&\le c_1 c_6 c_7 m_V(x),
\end{align*}
where we have again used Lemma \ref{ulDistance}(a) followed by Lemma \ref{partofU}(c),(d).
\end{proof}

\begin{rem}
\label{constantVRem}
To conclude this section, we consider these concepts when applied to the constant function $V_0(x) \equiv V_0 > 0$, which belongs to $RH_p$ for each $p > 1$ with $C_{V_0,p} = 1$.
Since $\psi_{V_0}(x, r) = V_0 r^2$ is constant in $x$, then $m_{V_0}(x) = \sqrt{V_0}$ is constant and we see that Lemma \ref{mBounds} holds with $c_1, c_2, c_3 = 1$ and $k_0 = 0$.
Moreover, the Agmon distance $d_{V_0}(x, y) = \sqrt{V_0} \abs{x - y}$ is a rescaling of the Euclidean distance and Lemma \ref{ulDistance} holds with $c_4, c_5 = 1$.
If we define $g_{V_0}(x, y) = \sqrt{V_0 \abs{x - y}^2 + 1}$, then $\abs{g_{V_0}(x, y) - d_{V_0}(x, y)} \le 1$ and
\begin{align*}
\abs{\gr_x g_{V_0}(x,y)} &= \frac{\sqrt{V_0} \abs{x - y}}{\sqrt{\abs{x - y}^2 + V_0^{-1}}} \le \sqrt{V_0} = m_{V_0}(x),
\end{align*}
showing that Lemma \ref{RegLem0} also holds with $c_8 = 1$.
With this notation, the inequality in Lemma \ref{MainUpperBoundLem0} now reads
\begin{align*}
\int_{\R^n} \abs{m_{V_0} \, u \, \phi }^2 e^{- 2 \ga g_{V_0}(\cdot,0) }
&\le C_\ga \int_{\R^n} \abs{u}^2 \abs{\gr \phi}^2 e^{-2 \ga g_{V_0}(\cdot,0)}.
\end{align*}
\end{rem}

\section{Fefferman-Phong Inequalities and Applications}
\label{FPI}

In this section, we present and prove a Fefferman-Phong inequality, then we use it establish our main tool, Lemma \ref{MainUpperBoundLem}.
In some sense, Lemma \ref{MainUpperBoundLem} serves as a replacement for a Carleman estimate in the proof of our main theorem.
The following Poincar\'e inequality will be used to prove the Fefferman-Phong inequality below.

\begin{lem}[Poincar\'e inequality, Lemma 0.14 in \cite{She99}]
\label{PoincareIneqThm}
If $V \in RH_{\frac{n}{2}}$, then there exists $c_P(n, C_V) > 0$ so that for any open ball $B \su \R^n$ of radius $r > 0$ and any $u \in C^1(B)$, it holds that
\begin{equation*}
{\int_B \int_B V(B)^{-1} V(y) \abs{u(x) - u(y)}^2 \, dx \, dy}
\le c_P r^2 {\int_B \abs{ \gr u(x)}^2 \, dx},
\end{equation*}
where $\disp V(B) := \int_B V(x) dx$.
\end{lem}

Next we present the Fefferman-Phong inequality and its proof.
Originally due to Fefferman and Phong \cite{Fef83}, a slight modification of the following version of the result can be found in \cite{She94}.
This estimate will be applied below to prove our main lemma.

\begin{prop}[Fefferman-Phong Inequality, Lemma 1.11 in \cite{She94}]
\label{FPml}
If $V \in RH_p$ for some $p \in (\frac n 2, \iny]$, then there exists $c_F(n, \la, p, C_V) > 0$ so for any $u \in C^1_0(\R^n)$, it holds that
$$\int_{\R^n} \abs{m_V(x) u(x)}^2  \, dx
\le c_F \pr{\frac \la 2 \inrn \abs{\gr u(x)}^2 \, dx + \inrn \abs{V^\frac12 (x) u(x)}^2 \, dx}.$$
\end{prop}

\begin{proof}
For $x_0 \in \R^n$, let $r_0 = m_V(x_0)^{-1}$ and set $B = B_{r_0}(x_0)$.
With $c_P$ from Lemma \ref{PoincareIneqThm}, we have
\begin{align*}
\int_{B} \abs{u(x)}^2  \, dx
&= \int_{B} \int_{B} V(B)^{-1} V(y) \abs{u(x)}^2 dy dx  \\
&\le \pr{1 +  \frac {\la} {2 c_P}} \int_{B} \int_{B} V(B)^{-1} V(y) \abs{u(x) - u(y)}^2 dy dx \\
&+  \pr{1 +  \frac {2 c_P}{\la} } \int_{B} \int_{B} V(B)^{-1} V(y) \abs{u(y)}^2 dy dx  \\
&\le \pr{1 + \frac {\la} {2 c_P}} c_P r_0^{2} \int_{B} \abs{\gr u(x)}^2 \, dx
+ \pr{1 + \frac {2 c_P}{\la} } \abs{B} \, V(B)^{-1} \int_{B} V(y) \abs{u(y)}^2 dy,
\end{align*}
where the last line follows from an application of Lemma \ref{PoincareIneqThm}.
Now we multiply this inequality through by $r_0^{-2} = m_V(x_0)^{2}$, then apply Lemma \ref{mBounds}(a) to conclude that 
\begin{align*}
\int_{B} \abs{m_V(x) u(x)}^2  \, dx
&\le c_1 \pr{1 + \frac {2 c_P} {\la}}  \pr{ \frac {\la} {2} \int_{B} \abs{\gr u(x)}^2 \, dx 
+  \int_{B} V(x) \abs{u(x)}^2 dx},
\end{align*}
since $r_0^{2} \abs{B}^{-1} V(B) = \psi_V(x_0, r_0) =1$. 

According to Lemma \ref{partofU}, there exists a sequence $\set{x_j}_{j=1}^\iny \su \R^n$ such that if we define $\disp B_j = B_{r_j}(x_j)$, where $r_j = m_V(x_j)^{-1}$, then $\disp \R^n = \bigcup_{j=1}^\iny B_j$ and $\disp \sum_{j=1}^\iny \chi_{B_j} \le c_7.$
Therefore, 
\begin{align*}
\int_{\R^n}\abs{m_V u}^2 
&\le \sum_{j=1}^\iny \int_{B_j} \abs{m_V u}^2 
\le \sum_{j=1}^\iny c_1 \pr{1 + \frac {2 c_P}{\la}}  \pr{\frac \la 2 \int_{B_j} \abs{\gr u}^2 +  \int_{B_j} V \abs{u}^2} \\
&\le c_1 c_7 \pr{1 + \frac {2 c_P}{\la} }  \pr{\frac \la 2 \int_{\R^n} \abs{\gr u}^2 + \int_{\R^n} V \abs{u}^2},
\end{align*}
as required.
\end{proof}

With Proposition \ref{FPml}, we now prove our main lemma.
Both the statement and the proof of this next result are inspired by \cite[Proposition 6.5]{MP19} which was used to establish upper bounds and exponential decay estimates for fundamental solutions.

\begin{lem}[Main Lemma]
\label{MainUpperBoundLem}
Let $A : \R^n \to \mathbb{M}_{n}$ be bounded and elliptic in the sense of \eqref{Abound} and \eqref{ellip} and let $V \in RH_p$ for some $p \in \pb{\frac n 2, \iny}$.
Assume that $u : \R^n \to \R$ is a weak solution to \eqref{ePDE} in the sense of \eqref{weakSol}.
Let $g \in C^1(\R^n)$ be a function for which $|\nabla g(x)| \le c_g m_V(x)$.
There exists $\ga_0(n, \la, \La, p, C_V, c_g) > 0$ such that whenever $\phi \in C_c ^\infty (\R^n)$ and $\ga < \ga_0$, it holds that
\begin{equation*}
\inrn \abs{m_V u \phi}^2  e^{-2 \ga g} \, \le c_L \inrn |u|^2 |\nabla \phi|^2 e^{-2\ga g},
\end{equation*}
where  $c_L(n, \la, \La, p, C_V, \ga_0, \ga) > 0$. 
\end{lem}

\begin{rem}
Following the discussion in Remark \ref{constantVRem}, we see that Lemma \ref{MainUpperBoundLem} generalizes Lemma \ref{MainUpperBoundLem0}.
\end{rem}

\begin{proof}
For some $\ga < \ga_0$ and $\phi \in C_c ^\infty (\R^n)$, let $f = \phi e^{-\ga g}$.
Set $\vp = u f^2$ and note that $\vp \in H^1_0(U)$, where $U := \supp \phi$.
Define $h = u f$ and note that $h \in H^1_0(U)$ as well.
Since $h^2 = u \vp$ and
\begin{align*}
A \gr h \cdot \gr h
&= A \pr{f \gr u + u \gr f} \cdot \pr{f \gr u + u \gr f} 
= f^2 A \gr u \cdot \gr u + u^2 A \gr f \cdot \gr f + u f \pr{A \gr u \cdot \gr f + A \gr f \cdot \gr u}  \\
A \gr u \cdot \gr \vp
&= A \gr u \cdot \gr\pr{u f^2}
= A \gr u \cdot \pr{f^2\gr u + 2 u f \gr f}
= f^2 A \gr u \cdot \gr u + 2 u f A \gr u \cdot \gr f,
\end{align*}
then
\begin{align*}
A \gr h \cdot \gr h
&= A \gr u \cdot \gr \vp
+ \abs{u}^2 A \gr f \cdot \gr f 
+ u f \pr{A^T - A} \gr u \cdot \gr f.
\end{align*}
Using ellipticity \eqref{ellip} and that $V \ge 0$, followed by the previous expression for $A \gr h \cdot \gr h$, we see that
\begin{equation*}
\begin{aligned}
\int \la \abs{\gr h}^2 + V \abs{h}^2
&\le \int A \gr h \cdot \gr h + V \abs{h}^2 \\
&= \int A \gr u \cdot \gr \vp +  V u \vp
+ \int \abs{u}^2 A \gr f \cdot \gr f + u f \pr{A^T - A} \gr u \cdot \gr f \\
&= \int \abs{u}^2 A \gr f \cdot \gr f + u f \pr{A^T - A} \gr u \cdot \gr f , 
\end{aligned}
\end{equation*}
where the last line uses \eqref{weakSol}. 
With the upper bound in \eqref{Abound}, we get that
\begin{equation}
\label{mainLemmaStep}
\begin{aligned}
\int \la \abs{\gr h}^2 + V \abs{h}^2
&\le \La \int \abs{u}^2 \abs{\gr f}^2 
+ 2 \La \int \abs{u} \abs{f} \abs{\gr u} \abs{\gr f}.
\end{aligned}
\end{equation}
Since
\begin{align*}
\abs{f}^2 \abs{\gr u}^2
&= \abs{f \gr u + u \gr f - u \gr f}^2
= \abs{\gr h - u \gr f}^2
\le \pr{1 + \frac{\la}{2\La} } \abs{\gr h}^2 + \pr{1 + \frac{2\La}{\la}} \abs{u}^2 \abs{\gr f}^2,
\end{align*}
then
\begin{align*}
2 \abs{u} \abs{f} \abs{\gr u} \abs{\gr f}
&\le \pr{\frac{\la + 2 \La}{\la}} \abs{u}^2 \abs{\gr f}^2
+ \pr{\frac{\la}{\la + 2 \La}}  \abs{f}^2 \abs{\gr u}^2  
\le \pr{\frac{2 \La}{\la} + 2} \abs{u}^2 \abs{\gr f}^2
+ \frac{\la}{2 \La}  \abs{\gr h}^2 .
\end{align*}
With $c_0 := \sqrt{\La  \pr{2\frac {\La}{\la} + 3}}$, the integral inequality in \eqref{mainLemmaStep} then simplifies to
\begin{align*}
\int \frac \la 2 \abs{\gr h}^2 + V \abs{h}^2
&\le c_0^2 \int \abs{u}^2 \abs{\gr f}^2.
\end{align*}

Because $\gr f = \pr{\gr \phi - \ga \gr g \phi} e^{- \ga g}$, then for any $\mu > 0$, Young's inequality shows that 
$$\abs{\gr f}^2 \le \brac{\pr{1 + \mu^{-1}}\abs{\gr \phi}^2 + \pr{1 + \mu} \ga^2 c_g^2 m_V^2 \abs{\phi}^2} e^{-2 \ga g},$$ 
where we have used the bound on $\gr g$.
Plugging this bound into the previous expression shows that
\begin{align*}
\int \frac \la 2 \abs{\gr h}^2 + V \abs{h}^2
&\le \pr{1 + \mu^{-1}} c_0^2 \int u^2 \abs{\gr \phi}^2 e^{-2 \ga g}
+ \pr{1 + \mu} \pr{c_0 \ga c_g}^2 \int \abs{m_V h}^2.
\end{align*}
Since $C^1_0(\R^n)$ is dense in $H^1(\R^n)$, then a limiting argument shows that we can apply Lemma \ref{FPml} to $h \in H^1_0(U)$ to get
\begin{align*}
\int \abs{m_V h}^2
&\le c_F \pr{\int \frac \la 2 \abs{\gr h}^2 + V h^2}.
\end{align*}
Combining the previous two inequalities shows that
\begin{align*}
\int \abs{m_V h}^2
&\le  \pr{1 + \mu^{-1}} c_F c_0^2  \int u^2 \abs{\gr \phi}^2 e^{-2 \ga g} 
+ \pr{1 + \mu} c_F \pr{c_0 c_g}^2 \ga^2 \int \abs{m_V h}^2.
\end{align*}

If $\ga  <  \ga_0 := \frac{1}{c_0 c_g\sqrt{c_F} }$, choose $\mu =  \frac{\ga_0^2 - \ga^2}{2\ga^2} > 0$ so that 
$$\pr{1 + \mu} c_F \pr{c_0 c_g}^2 \ga^2 = c_F \pr{c_0 c_g}^2 \frac{\ga_0^2 + \ga^2}2 = \frac 1 2\brac{1 + c_F \pr{c_0 c_g \ga}^2}.$$
After simplifications, we get that
\begin{align*}
\int \abs{m_V u \phi}^2  e^{-2 \ga g}
= \int \abs{m_V h}^2
&\le c_F \brac{\frac{2 c_0}{1 - c_F \pr{c_0 c_g \ga}^2}}^2   \int u^2 \abs{\gr \phi}^2 e^{-2 \ga g},
\end{align*}
as required.
\end{proof}

\section{Proofs of the Main Results}
\label{MainResultProofs}

In this section, we prove Theorem \ref{MainThm}, Corollary \ref{LiouvilleCor}, and Corollary \ref{OGICor}.
For Theorem \ref{MainThm}, the idea is to proceed by contradiction and apply Lemma \ref{MainUpperBoundLem} to our solution function with an appropriate choice of cutoff function.
We choose cutoff functions with gradients that are supported in the sets over which \eqref{uIntBound} holds.
By taking a limit, we reach our desired contradiction.

We prove Corollary \ref{LiouvilleCor} and Corollary \ref{OGICor} by appealing to the properties of $d_V$ and $m_V$ outlined in Section \ref{RHetc}.

\begin{proof}[Proof of Theorem \ref{MainThm}]

Assume for the sake of contradiction that $u \not\equiv 0$ on $\R^n$.
In this case, there exists $R_0 \ge m_V(0)^{-1}$ so that with $B_{R_0} := B_{R_0}(0)$, 
\begin{equation}
\label{nontrivial}
\norm{u}_{L^2(B_{R_0})} = c_0 > 0.
\end{equation}

Let $\ga_0 = \ga_0(n, \la, \La, p, C_V, c_8)$ be from Lemma \ref{MainUpperBoundLem}, where $c_8$ is from Lemma \ref{RegLem0}.
Let $\set{R_j}_{j = 1}^\iny \su \R_+$ be the unbounded, increasing sequence for which \eqref{uIntBound} holds with some $\ga < \ga_0$.
Without loss of generality, $R_j \ge R_0$ for each $j \in \N$.

For each $j \in \N$, let $\phi_j \in C_c ^\infty(B_{R_j+1}(0))$ be a cutoff function with $\phi_j \equiv 1$ in the ball $B_{R_j} := B_{R_j}(0)$ and $|\nabla \phi_j | \leq 2$ in the annulus $A(R_j, R_j + 1) := A(0, R_{j}, R_{j} + 1)$.

With $\ga$ from \eqref{uIntBound}, $\phi = \phi_j$, and $g = g_V\pr{\cdot, 0}$, where $g_V \in C^\iny(\R^n)$ is from Lemma \ref{RegLem0}, we apply Lemma \ref{MainUpperBoundLem} to get
\begin{equation}
\label{lemmaApp}
\begin{aligned}
\int_{B_{R_j}} \abs{m_V u}^2 e^{-2 \ga g_V\pr{\cdot, 0}}
&\le \int_{\R^n} \abs{m_V u \phi_j}^2 e^{-2\ga g_V\pr{\cdot, 0}} 
\le c_L \int_{\R^n} | u|^2  |\nabla \phi_j |^2  e^{-2\ga g_V\pr{\cdot, 0}} \\
&\le 4 c_L \int_{A(R_j, R_j + 1)} |u| ^2 e^{-2\ga g_V\pr{\cdot, 0}}
\le 4 c_L e^{2 \ga c_1} \int_{A(R_j, R_j + 1)} |u| ^2 e^{-2\ga d_{V}\pr{\cdot, 0}},
\end{aligned}
\end{equation}
where the last bound follows from an application of the first bound in Lemma \ref{RegLem0}.

To bound the integral over $B_{R_j}$, we use that $B_{R_j} \supset B_{R_0}$.
For each $x \in B_{R_0}$, Lemma \ref{mBounds}(c) shows that 
$$m_V(x) \ge \frac{c_3 m_V(0)}{\pr{1 + \abs{x} m_V(0)}^{k_0 /\pr{k_0 + 1}}} \ge \frac{c_3 m_V(0)}{\pr{1 + R_0 \, m_V(0)}^{k_0 /\pr{k_0 + 1}}},$$ 
while Lemma \ref{ulDistance}(a)(b) shows that 
$$d_V(x, 0) \le c_1 + c_4 \pr{ R_0 \, m_V(0)}^{k_0 + 1}.$$
In combination with Lemma \ref{RegLem0}, we see that for each $x \in B_{R_0}$,
$$g_V(x, 0) \le 2 c_1 + c_4 \pr{R_0 \, m_V(0)}^{k_0 + 1}.$$ 
Therefore, 
\begin{equation}
\label{lowerBound}
\begin{aligned}
\int_{B_{R_j}} \abs{m_V u}^2 e^{-2\ga g_V\pr{\cdot, 0}}
&\ge \int_{B_{R_0}} \abs{m_V u}^2 e^{-2\ga g_V\pr{\cdot, 0}} \\
&\ge \brac{\frac{c_3 m_V(0)}{\pr{1 + R_0 \, m_V(0)}^{k_0 /\pr{k_0 + 1}}} e^{-\ga\pr{2 c_1 + c_4 \pr{R_0 \, m_V(0)}^{k_0 + 1}}} }^2 \int_{B_{R_0}}  |u(x)|^2 dx \\
&= \brac{\frac{c_0 c_3 m_V(0)}{\pr{1 + R_0 \, m_V(0)}^{k_0 /\pr{k_0 + 1}}} e^{-\ga\pr{2 c_1 + c_4 \pr{R_0 \, m_V(0)}^{k_0 + 1}}} }^2,
\end{aligned}
\end{equation}
where we have used the assumption in \eqref{nontrivial}.

Substituting \eqref{lowerBound} into \eqref{lemmaApp} and simplifying shows that for each $j \in \N$,
\begin{equation*}
\begin{aligned}
\brac{\frac{c_0 c_3 m_V(0) }{2 \sqrt{c_L} \pr{1 + R_0 \, m_V(0)}^{k_0 /\pr{k_0 + 1}}} e^{-\ga\pr{3 c_1 + c_4 \pr{R_0 \, m_V(0)}^{k_0 + 1}}}}^2
&\le \int_{A(R_j, R_j + 1)} |u(x)| ^2 e^{-2\ga d_{V}\pr{x, 0}} dx.
\end{aligned}
\end{equation*}
By taking $j \to \iny$ and using \eqref{uIntBound}, we conclude that the lefthand side is zero.
However, this is only possible if $c_0 = 0$, which gives our desired contradiction.
\end{proof}

Next we show how the Liouville-type statement in Corollary \ref{LiouvilleCor} follows from Theorem \ref{MainThm}.

\begin{proof}[Proof of Corollary \ref{LiouvilleCor}]
It suffices to show that \eqref{uPwBound} implies \eqref{uIntBound} for some sequence $\set{R_j}_{j=1}^\iny$ and some $\ga < \ga_0$.
Indeed, let $R_j= j$ for each $j \in \N$ and set $\ga = \frac 1 2 \pr{\ga_1 + \ga_0}$.
An application of \eqref{uPwBound} shows that
\begin{align*}
\int_{A(j, j+1)} \abs{u(x)}^2 e^{- 2 \ga d_V(x, 0)} dx
&\le C_{\ga_1}^2 \int_{A(j, j+1)} e^{- \pr{\ga_0 - \ga_1} d_V(x, 0)} dx.
\end{align*}
If $j \ge m_V(0)^{-1}$, then Lemma \ref{ulDistance}(c) shows that for each $x \in A(j, j+1)$,
\begin{align*}
d_V(x, 0) 
&\ge c_5 \pr{\abs{x } m_V(0) }^{1 / \pr{k_0 + 1}} 
\ge c_5 \pr{j \, m_V(0) }^{1 /\pr{k_0 + 1}}.
\end{align*}
Since $\abs{A(j, j+1)} \le c_n j^{n-1}$, then
\begin{align*}
\lim_{j \to \iny} \int_{A(j, j+1)} \abs{u(x)}^2 e^{- 2 \ga d_V(x, 0)} dx
&\le \lim_{j \to \iny} \frac{C_{\ga_1}^2 c_n j^{n-1}}{\exp\set{\pr{\ga_0 - \ga_1}c_5 \pr{j \, m_V(0) }^{1 /\pr{k_0 + 1}} }} 
= 0,
\end{align*}
which establishes \eqref{uIntBound}.
\end{proof}

Finally, the contrapositive of Theorem \ref{MainThm} in combination with technical properties of the distance function (and its related quantities) leads to Corollary \ref{OGICor}.

\begin{proof}[Proof of Corollary \ref{OGICor}]
If $u \not\equiv 0$, then Theorem \ref{MainThm} shows that for every $\ga < \ga_0$, 
\begin{equation*}
\lim_{j \to \iny} \int_{A\pr{j , j + 1}} \abs{u(x)}^2 e^{-2 \ga d_V(x, 0)} dx \ne 0.
\end{equation*}
Fix $\ga_1 < \ga_0$ and take $\ga = \frac{\ga_1 + \ga_0}{2} < \ga_0$.
There exists a constant $c_0 > 0$ and an increasing, unbounded sequence $\set{R_j}_{j = 1}^\iny \su \N$ with the property that for every $j \in \N$,
\begin{equation}
\label{contraCons}
\int_{A\pr{R_j, R_j + 1}} \abs{u(x)}^2 e^{-2 \ga d_V(x, 0)} \ge c_0.
\end{equation}
 
Fix $j \in \N$ for which $R_j \ge \max\set{m_V(0)^{-1} + 1, \tfrac 1 2 c_2^{-\frac 1{k_0}} m_V(0)^{-1 - \frac 1{k_0}}, c_0^{-1} \pr{3 c_2 2^{k_0}}^n e^{\pr{\ga_0 + \ga_1}c_1} }$, where $c_1$, $c_2$, and $k_0$ are from Lemma \ref{mBounds}.
For each $x \in \R^n$, define $r(x) = \min\set{1, m_V(x)^{-1}}$.
Since $\overline{A(R_j, R_j +1)}$ is compact, a Vitali covering argument shows that there exists a finite set of points, $\set{x_i}_{i = 1}^N \su A(R_j, R_j + 1)$, so that $\set{B_{r(x_i)/3}(x_i)}_{i = 1}^N$ is disjoint and $\set{B_{r(x_i)}(x_i)}_{i = 1}^N$ covers $A(R_j, R_j+1)$.
As $\set{B_{r(x_i)}(x_i)}_{i = 1}^N$ covers $A(R_j, R_j+1)$, then from \eqref{contraCons}, we get that
\begin{align*}
c_0
&\le \int_{A\pr{R_j, R_j + 1}} \abs{u(x)}^2 e^{-2 \ga d_V(x, 0)} 
\le \int_{\bigcup B_{r(x_i)}(x_i)} \abs{u(x)}^2 e^{-2 \ga d_V(x, 0)} 
\le \sum_{i=1}^N \int_{B_{r(x_i)}(x_i)} \abs{u(x)}^2 e^{-2 \ga d_V(x, 0)} .
\end{align*}
After relabelling, the pigeonhole principle implies that
\begin{align}
\label{pigeon}
\int_{B_{r(x_1)}(x_1)} \abs{u(x)}^2 e^{-2 \ga d_V(x, 0)} 
\ge \frac{c_0}{N}.
\end{align}
For each $x_i$, Lemma \ref{mBounds}(b) shows that
$$m_V(x_i) \le c_2 \pr{1 + \abs{x_i} m_V(0)}^{k_0} m_V(0) \le c_2 \pr{1 + \pr{R_j + 1} m_V(0)}^{k_0} m_V(0) \le c_2 (2R_j)^{k_0} m_V(0)^{k_0 + 1},$$
where the last inequality uses that $m_V(0)^{-1} \le R_j - 1$.
Since $c_2 (2R_j)^{k_0} m_V(0)^{k_0 + 1} \ge 1$, we deduce that $r(x_i) \ge \brac{c_2 (2R_j)^{k_0} m_V(0)^{k_0 + 1}}^{-1}$.
Because $\set{B_{r(x_i)/3}(x_i)}_{i = 1}^N \su A(R_j - \tfrac 1 3, R_j + \tfrac 4 3)$ is disjoint, we have
\begin{align*}
N \abs{\mathbb{B}_1} \brac{c_2 (2R_j)^{k_0} m_V(0)^{k_0 + 1}}^{-n}
&\le \abs{\mathbb{B}_1}  \sum_{i=1}^N r(x_i)^n
= \sum_{i=1}^N \abs{B_{r(x_i)}(x_i)}
= 3^{n} \sum_{i=1}^N \abs{B_{r(x_i)/3}(x_i)} \\
&= 3^{n}  \abs{\bigsqcup_{i=1}^N B_{r(x_i)/3}(x_i)} 
\le 3^{n}  \abs{A(R_j - \tfrac 1 3, R_j + \tfrac 4 3)}
\le c_n 3^{n}  \abs{\mathbb{B}_1} R_j^{n-1} \\
&\le \frac{c_0 c_n \abs{\mathbb{B}_1}}{\pr{c_2 2^{k_0}}^n e^{\pr{\ga_0 + \ga_1}c_1}}   R_j^{n},
\end{align*}
where we have used the lower bound on $R_j$ in the final inequality.
Simplifying this expression shows that $\disp N \le \frac{c_0 c_n}{e^{\pr{\ga_0 + \ga_1}c_1}} \pr{ R_j  m_V(0)}^{n(k_0 + 1)}$, which we substitute into \eqref{pigeon} to get
\begin{align*}
\int_{B_{r(x_1)}(x_1)} \abs{u(x)}^2 e^{-2 \ga d_V(x, 0)} 
\ge \frac{e^{\pr{\ga_0 + \ga_1}c_1}}{ c_n \pr{ R_j  m_V(0)}^{n(k_0 + 1)}}.
\end{align*}
As $r(x_1) \le m_V(x_1)^{-1}$, then Lemma \ref{ulDistance}(a) shows that $d_V(x, 0) \ge d_V(x_1, 0) - d_V(x, x_1) \ge d_V(x_1, 0) - c_1$ for each $x \in B_{r(x_1)}(x_1)$.
It follows that
\begin{align*}
\int_{B_{r(x_1)}(x_1)} \abs{u(x)}^2 
&\ge \frac{e^{\pr{\ga_0 + \ga_1}c_1}}{ c_n e^{2 \ga c_1} \pr{ R_j  m_V(0)}^{n(k_0 + 1)}} e^{2 \ga d_V(x_1, 0)}
= \frac{\exp\brac{(\ga_0 - \ga_1) d_V(x_1, 0)}}{c_n \pr{ R_j  m_V(0)}^{n(k_0 + 1)}} e^{2 \ga_1 d_V(x_1, 0)},
\end{align*}
where we recall our choice of $\ga$.
An application of Lemma \ref{ulDistance}(c) shows that 
\begin{align*}
d_V(x_1, 0) 
&\ge c_5 \pr{\abs{x_1} m_V(0) }^{1 / \pr{k_0 + 1}} 
\ge c_5 \pr{R_j \, m_V(0) }^{1 /\pr{k_0 + 1}}
\end{align*}
and then, since $r(x_1) \le 1$, we get
\begin{align*}
\int_{B_{1}(x_1)} \abs{u(x)}^2 
&\ge \frac{\exp\brac{(\ga_0 - \ga_1) c_5 \pr{R_j \, m_V(0) }^{1 /\pr{k_0 + 1}}}}{c_n \pr{ R_j  m_V(0)}^{n(k_0 + 1)}} e^{2 \ga_1 d_V(x_1, 0)}.
\end{align*}
Since $\disp f(x) = e^{\mu x} x^{-\al}$ grows without bound in $x$ for any $\mu, \al > 0$, then for each $M > 0$, there exists $j \gg 1$ and $x_j \in A(R_j, R_j+1)$ so that
\begin{align*}
\int_{B_1(x_j)} \abs{u(x)}^2
\ge M e^{2 \ga_1 d_V(x_j, 0)}.
\end{align*}
As this holds for every $\ga_1 < \ga_0$, then the conclusion described by \ref{uballBound} follows.
\end{proof}

\bibliography{refs}
\bibliographystyle{alpha}

\end{document}